\documentclass[reqno,12pt]{amsart}

\title
[Self-expanders of the mean curvature flow]
{Self-expanders of the mean curvature flow}
\author[Knut Smoczyk]{\sc Knut Smoczyk}
\address{
Institute of Differential Geometry\newline
Leibniz University Hannover, Germany}
\email{smoczyk@math.uni-hannover.de}
\subjclass[2010]{Primary 53C44, 53C21, 53C42}
\keywords{%
	mean curvature flow, self-expander}
\thanks{The author was supported by the German Research Foundation within the priority program SPP 2026 - Geometry at Infinity, DFG SM 78/7-1. }
\date{\today} 

\usepackage{amscd,amsfonts,mathrsfs,amsthm,enumerate}
\usepackage{amssymb, amsmath}          
\usepackage{stmaryrd}                          
\usepackage{epsfig}
\usepackage[shortalphabetic]{amsrefs}    
\def\real     #1{{\mathbb R^{#1}}}
\def\complex  #1{{\mathbb C^{#1}}}

\def\ds       {\nabla\hspace{-1.5pt}s}

\def\equationcolor {\color{black}}
\def\textcolor     {\color{black}}

\def\bcoleq    {\begin{equation}\equationcolor}
\def\ecoleq    {\textcolor\end{equation}}

\def\bcoleqn   {\equationcolor\begin{eqnarray}}
\def\ecoleqn   {\end{eqnarray}\textcolor}

\def\herv #1 {{\rm{#1} }}

\newtheorem{theorem}{Theorem}[section]
\newtheorem{lemma}[theorem]{Lemma}
\newtheorem{corollary}[theorem]{Corollary}
\newtheorem{proposition}[theorem]{Proposition}
\newtheorem{remark}[theorem]{Remark}
\theoremstyle{definition}

\newcommand{\bfig}{\begin{figure}}
\newcommand{\efig}{\end{figure}}

\makeatletter
\def\pproof#1{\@ifnextchar[\opargproof
{\opargproof[\it Proof of #1.]}}
\def\opargproof[#1]{\par\noindent {\bf #1 }}

\makeatother

\begin{document}
\maketitle
\begin{abstract}
	We study self-expanding solutions $M^m\subset\real{n}$ of the mean curvature flow. One of our main results is, that complete mean convex self-expanding hypersurfaces are products of self-expanding curves and flat subspaces, if and only if the function $|A|^2/|H|^2$ attains a local maximum, where $A$ denotes the second fundamental form and  $H$ the mean curvature vector of $M$. If the pricipal normal $\xi=H/|H|$ is parallel in the normal bundle, then a similar result holds in higher codimension for the function $|A^\xi|^2/|H|^2$, where $A^\xi$ is the second fundamental form with respect to $\xi$. As a corollary we obtain that complete mean convex self-expanders attain strictly positive scalar curvature, if they are smoothly asymptotic to cones of non-negative scalar curvature. In particular, in dimension $2$ any mean convex self-expander that is asymptotic to a cone must be strictly convex.
\end{abstract}

\section{Introduction}\label{introduction}
A smooth immersion $F:M\to\real{n}$ of a manifold $M$ of dimension $m$ into euclidean space is called a
\textit{self-expander}
of the mean curvature flow, if it satisfies the equation
\begin{gather}
	H=\lambda F^\perp,\tag{$\ast$}\label{main 1}
\end{gather}
where $H$ is the mean curvature vector of the immersion, $\lambda$ is a positive constant and where $^\perp$ denotes the orthogonal projection onto the normal bundle of $M$.
\newline

Self-expanders arise naturally when one considers solutions of graphical mean curvature flow. In the case of codimension $1$ and under certain assumptions on the initial hypersurface at infinity, Ecker and Huisken \cite{eh} showed that the solutions of mean curvature flow of entire graphs in euclidean space exist for all times $t>0$ and become asymptotically self-expanding as $t\to\infty$. Later Stavrou \cite{stavrou} proved such a result under the weaker assumption that the initial hypersurface attains a unique tangent cone at infinity. Rasul \cite{rasul} showed that under an alternative condition at infinity and bounded gradient, the rescaled graphs converge to self-similar solutions but at a slower speed. Clutterbuck and Schn\"urer \cite{clutsch} considered graphical solutions to mean curvature flow and obtained a stability result for homothetically expanding solutions coming out of cones of positive mean curvature. It is expected that similar results hold for the mean curvature flow in higher codimension of entire graphs generated by contractions and area decreasing maps as studied in \cite{ss1}, \cite{ss2}, \cite{ss3}.

As was pointed out in \cite{eh} and \cite{stavrou}, self-expanders also arise as solutions of the mean curvature flow, if the initial submanifold is a cone. Moreover, in some situations uniqueness of self-expanders is important for the construction of mean curvature flows starting from certain singular configurations \cite{bm}. Fong and McGrath \cite{fm} proved a Liouville-type theorem for complete, mean-convex self-expanders whose ends have decaying principal curvatures. Ding \cite{ding} studied self-expanding solutions and their relationship to minimal cones. The space of asymptotically conical self-expanders was studied in several papers by Bernstein and Wu, as for example in \cite{bw1} and \cite{bw2}.

Cheng and Zhou \cite{cz} proved results for self-expanders in higher codimension related to the spectrum of the drifted Laplacian. In higher codimension self-expanders have been studied in particular for the Lagrangian mean curvature flow. In \cite{lee-wang}, \cite{jlt} and \cite{nakahara} new examples of Lagrangian self-expanders were given. Lotay and Neves \cite{ln} proved that zero-Maslov class Lagrangian self-expanders in $\complex{n}$ that are asymptotic to a pair of planes intersecting transversely are locally unique for $n>2$ and unique if $n=2$. Further uniqueness results for Lagrangian self-expanders asymptotic to the union of two transverse Lagrangian planes were shown by Imagi and Joyce \cite{ij}. 

Many of the above mentioned results show that the geometry of a self-expander is strongly determined by its asymptotic structure at infinity. This is confirmed also by the main results of this paper. We show that the \textit{pinching quantity} $|A^H|^2/|H|^4$, where $A^H$ denotes the second fundamental form with respect to the mean curvature vector $H$ is controlled by its geometry at infinity. In particular, this implies a number of uniqueness results for self-expanders with a certain asymptotic behavior. Since self-expanders are also minimal submanifolds with respect to a conformally flat Riemannian metric on $\real{n}$, the analysis of self-expanders is very similar to that of classical minimal submanifolds. Therefore one expects also Bernstein type theorems for self-expanders similar to the classical Bernstein theorems in higher codimension, for example as derived in \cite{jxy1}, \cite{jxy2}, \cite{jx}, \cite{ss0}, \cite{swx}.

\section{Structure equations for general euclidean submanifolds}\label{section 2}
Let $ F:M^m\to \real{n}$ be a smooth immersion. We denote its
\textit{pullback bundle}
by $F^{*}T\real{n}$ and the
\textit{normal bundle}
by $T^\perp M$. The induced
\textit{metric}
or 
\textit{first fundamental form}
$g=F^*\langle\cdot,\cdot\rangle$ on $TM$ is given by
$$g(v,w)=\langle dF(v),dF(w)\rangle,\quad\forall v,w\in TM,$$
where $dF\in\Gamma(F^{*}T\real{n}\otimes T^*M)$ denotes the
\textit{differential}
of $F$.
The
\textit{second fundmental form} 
$A\in\Gamma(F^{*}T\real{n}\otimes T^*M\otimes T^*M)$ is defined by
$$
A:=\nabla dF.
$$
Here and in the following all canonically induced full Levi-Civita connections on
product bundles over $M$ will be denoted by $\nabla$. We will sometimes use the connection on the normal bundle and on bundles formed from products with the normal bundle. These connections will be denoted by $\nabla^\perp$. 

By definition, $A$ is a section in the bundle $ F^{*}T\real{n}\otimes T^*M\otimes T^*M$ but it is well known that $ A$ is normal, i.e.
$$
A\in\Gamma\left(T^\perp M\otimes T^*M\otimes T^*M\right).
$$
This implies that
\begin{equation}\nonumber
	\langle A(v,w),dF(z)\rangle=0,\quad\forall v,w,z\in T_pM.
\end{equation}
The
\textit{mean curvature vector field}
$H\in\Gamma(T^\perp M)$ is the trace of the second fundamental tensor. At $p\in M$ we have
$$H=\operatorname{trace}_g(A)=\sum_{k=1}^mA(e_k,e_k),$$
where $(e_k)_{k=1,\dots,m}$ denotes an arbitrary orthonormal basis of $T_pM$ (it should be noted that many authors prefer to define $H$ as $\frac{1}{m}\operatorname{trace}_gA$, for example this is done in \cite{jost}).

For any normal vector $\xi\in T^\perp M$ we define the second fundamental
form $A^\xi$ and the scalar mean curvature $H^\xi$ with respect to $\xi$ by
$$A^\xi(v,w):=\langle A(v,w),\xi\rangle\,,\quad H^\xi:=\langle H,\xi\rangle=\operatorname{trace}_g(A^\xi).$$

The Riemannian curvature tensor on the tangent bundle will be denoted by $R$, whereas the curvature tensor of the normal bundle, considered as a 2-form with values in $\operatorname{End}(T^\perp M)$ will be written as $R^\perp$.

We summarize the equations of \textsc{Gau\ss}, \textsc{Ricci}, \textsc{Codazzi} and \textsc{Simons} in the following proposition.
\begin{proposition}\label{prop 1}
	Let $M$ be an $m$-dimensional smooth manifold and $F:M\to\real{n}$ be a smooth immersion. Then for any $p\in M$ and any $\xi\in T_p^\perp M$, $v,w,u,z\in T_pM$ we have
	\begin{enumerate}[\normalfont(a)]
		\item  \textsc{Gau\ss:}
		$$R(v,w,u,z)=\langle A(v,u),A(w,z)\rangle-\langle A(v,z),A(w,u)\rangle.$$
		\item  \textsc{Ricci:}
		$$R^\perp(v,w)\xi=\sum_{k=1}^m
		\left(A^\xi(w,e_k) A(v,e_k)-A^\xi(v,e_k)A(w,e_k)\right),$$
		where $(e_k)_{k=1,\dots,m}$ is an orthonormal basis of $T_pM$.
		\item \textsc{Codazzi:} 
		$$\nabla^\perp_vA(w,z)=\nabla^\perp_wA(v,z).$$
		\item \textsc{Simons:} 
		\begin{gather*}
			\Delta^\perp A(v,w)=(\nabla^\perp)^2_{v,w}H+\sum_{k=1}^mA^H(v,e_k)A(w,e_k)\\
			+2\sum_{k,l=1}^m\langle A(v,e_k),A(w,e_l)\rangle A(e_k,e_l)-\sum_{k,l=1}^m\langle A(v,w),A(e_k,e_l)\rangle A(e_k,e_l)\\
			-\sum_{k,l=1}^m\langle A(v,e_k),A(e_k,e_l)\rangle A(w,e_l)-\sum_{k,l=1}^m\langle A(w,e_k),A(e_k,e_l)\rangle A(v,e_l).
		\end{gather*}
	\end{enumerate}
\end{proposition}
We set
$$r:=|F|,\quad s:=\frac{r^2}{2}.$$
Decomposing $F=F^\perp+F^\top$ into its normal and tangent components, the following equations are well known (see for example \cite{smoczyk}):
\begin{equation}\label{f1}
	F^\top=\ds,\quad \nabla^\perp_v F^\perp=-A(\nabla s,v),\quad \nabla^2 s=g+\langle F,A\rangle
\end{equation}
and
\begin{equation}
	(\nabla^\perp)^2_{v,w}F^\perp=-A(v,w)-\nabla^\perp_{\ds}A(v,w)-\sum_{k=1}^mA^{F^\perp}(v,e_k)A(w,e_k).\label{f2}
\end{equation}

\section{Geometric equations for self-expanders}\label{section 3}
\begin{remark}\label{rmk 1} We give some remarks.
	\begin{enumerate}[\rm (a)]
		\item A self-expander that is minimal must be a cone. Hence, if the self-expander is everywhere smooth and complete, then the only minimal self-expanders are given by linear subspaces  $U\subset \real{n}$.
		\item Since the Riemannian product of two self-expanders with the same constant $\lambda$ is again a self-expander, one observes that self-expanders are not necessarily asymptotic to cones. For example the product of an expanding curve $\Gamma$ with a linear subspace is such a special case. These self-expanders will become important further below.
	\end{enumerate}
\end{remark}

Let us now assume that the immersion is a self-expander, i.e. 
$$H=\lambda F^\perp$$
for some positive\footnote{ Most computations will hold as well for $\lambda<0$, i.e. for self-shrinkers.} constant $\lambda$. From \eqref{f1} and \eqref{f2} we derive the following equations:
\begin{equation}\label{self 2}
	\nabla_v^\perp H=-\lambda A(\ds,v)
\end{equation}
and
\begin{equation}
	(\nabla^\perp)^2_{v,w} H=-\lambda A(v,w)
	-\lambda\nabla^\perp_{\ds}A(v,w)-\sum_{k=1}^mA^H(v,e_k)A(w,e_k).\label{self 3}
\end{equation}
Taking a trace gives
\begin{equation}\label{self 4}
	\Delta^\perp H+\lambda\nabla^\perp_{\ds}H+\sum_{k,l=1}^mA^H(e_k,e_l) A(e_k,e_l)+\lambda H=0.
\end{equation}\label{self 4b}
If we take a scalar product with $2H$ we obtain
\begin{equation}\label{self 5}
	\Delta|H|^2-2|\nabla^\perp H|^2+\lambda\langle\nabla s,\nabla|H|^2\rangle+2|A^H|^2+2\lambda|H|^2=0.
\end{equation}
In addition, combining Simons' identity with \eqref{self 3} we get
\begin{gather}
	\Delta^\perp A(v,w)
	=-\lambda A(v,w)
	-\lambda\nabla^\perp_{\ds}A(v,w)\nonumber\\
	+2\sum_{k,l=1}^m\langle A(v,e_k),A(w,e_l)\rangle A(e_k,e_l)-\sum_{k,l=1}^m\langle A(v,w),A(e_k,e_l)\rangle A(e_k,e_l)\nonumber\\
	-\sum_{k,l=1}^m\langle A(v,e_k),A(e_k,e_l)\rangle A(w,e_l)-\sum_{k,l=1}^m\langle A(w,e_k),A(e_k,e_l)\rangle A(v,e_l).\label{self 6}
\end{gather}
\section{Self-expanding curves}\label{section 4}
In the case of curves $\Gamma\subset\real{n}$ the equation for self-expanders becomes a 2nd order system of ODEs. The existence and uniqueness theorem of \textsc{Picard--Lindel\"of} implies that for any $p,w\in\real{n}$, $w\neq 0$, there exists a maximal open interval $I$ containing $0$ and a uniquely determined curve $\Gamma:I\to\real{n}$, parameterized proportional to arc-length, such that
$$\overrightarrow k=\lambda \Gamma^\perp,\quad \Gamma(0)=p,\quad \Gamma'(0)=w,$$
where $\overrightarrow k$ denotes the curvature vector of $\Gamma$. If $p,w$ are collinear, then clearly the straight lines passing through the origin in direction of $w$ are the solutions. On the other hand, if $p,w$ are linearly independent, then again by the uniqueness part in the theorem of \textsc{Picard--Lindel\"of} the solution $\Gamma$ must be a planar curve in the plane spanned by $p, w$. Since rotations around the origin map self-expanders to self-expanders, one may therefore without loss of generality consider only self-expanding curves in $\real{2}$. For these curves the equation for self-expanders becomes
\begin{equation}\label{self curve}
	k=\lambda\langle \Gamma,\xi\rangle,
\end{equation}
where $\xi$ is the unit normal along $\Gamma$ obtained by rotating $\Gamma'/\vert\Gamma'\vert$ to the left by $\pi/2$ and $k$ denotes the curvature function of $\Gamma$ defined by $\overrightarrow k=k\xi$. Solutions of \eqref{self curve} have been studied in great detail and are completely classified. Ishimura \cite{ishimura} showed that self-expanding curves are asymptotic to a cone with vertex at the origin. It is easy to show (see Theorem 3.20 in \cite{gssz} or Lemma 6.4 in \cite{halldorsson}) that the function $ke^{\frac{\lambda}{2}r^2}$ is constant along $\Gamma$, where $r:=\vert\Gamma\vert$. The following description of self-expanding curves can be found in \cite{halldorsson}:
\begin{proposition}[Halldorsson]
	All self-expanding curves $\Gamma\subset\real{2}$ are convex, properly embedded and asymptotic to the boundary of a cone with vertex at the origin. They are graphs of even functions and form a one-dimensional family parametrized by their distance $r_0$ to the origin, which can take on any value in $[0,\infty)$.
\end{proposition}
\begin{figure}
	\begin{center}
		\includegraphics[width=.65\textwidth]{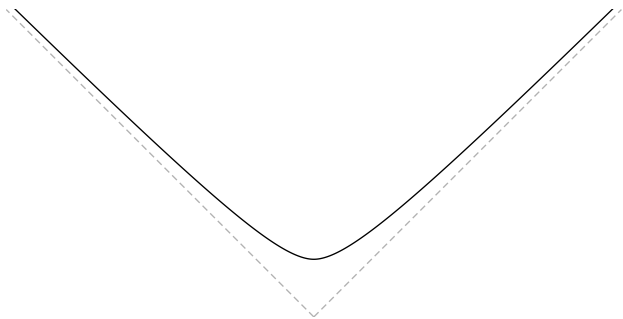}
	\end{center}
	\caption{A self-expanding curve is asymptotic to a cone.}
	\label{fig selfexpander}       
\end{figure}
Moreover, the total curvature $\int_\Gamma k$ of the curves is given by $\pi-\alpha$, where $\alpha$ denotes the opening angle of the asymptotic cone.

In the sequel, any self-expander $M=\Gamma\times\real{m-1}\subset\real{m+1}$, where $\Gamma$ is a non-trivial self-expanding curve in $\real{2}$ (i.e. not a straight line), will be called a \textit{self-expanding hyperplane} (see figure \ref{fig hyperplane}). In particular self-expanding hyperplanes are diffeomorphic to $\real{m}$.
%

\section{Mean convex self-expanding hypersurfaces}\label{section 5}
In this section we will consider complete connected self-expanding hypersurfaces. From equation \eqref{main 1} one observes that a smooth minimal self-expander must be totally geodesic, hence a linear subspace. One of our main theorems is:
\begin{theorem}\label{main theo1}
	Let $M^m\subset\real{m+1}$ be a smooth and complete connected self-expander that is different from a linear subspace. Then the set $\{H\neq 0\}$ is non-empty and the following statements are equivalent:
	\begin{enumerate}[\rm (a)]
		\item $M$ is a self-expanding hyperplane $\Gamma\times\real{m-1}$.
		\item The function $\frac{\vert A\vert^2}{\vert H\vert ^2}$ attains a local maximum on the open set $\{H\neq 0\}$.
	\end{enumerate}
	If one of these equivalent conditions is satisfied, then the set $\{H=0\}$ is empty and the function $\frac{|A|^2}{|H|^2}$ is constant to $1$.
\end{theorem}
\begin{figure}
	\begin{center}
		\includegraphics[width=.75\textwidth]{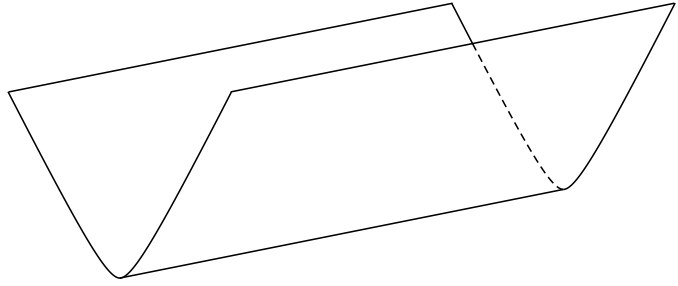}
	\end{center}
	\caption{Part of the $2$-dimensional self-expanding hyperplane $\Gamma\times\real{}\subset\real{3}$ that is asymptotic to the product of a cone $C\subset\real{2}$ and a real line.}
	\label{fig hyperplane}       
\end{figure}
\begin{proof}
	Since $M$ is a hypersurface we have $\vert A^H\vert ^2=\vert H\vert ^2\vert A\vert^2$ and \eqref{self 5} becomes
	\begin{equation}\label{self 5h}
		\Delta|H|^2-2|\nabla^\perp H|^2+\lambda\langle\nabla s,\nabla|H|^2\rangle+2(\vert A\vert ^2+\lambda)|H|^2=0.
	\end{equation}
	Moreover, on the set where $H$ is non-zero we have $|\nabla^\perp H|^2=\vert\nabla\vert H\vert\vert^2$.
	In addition, \eqref{self 6} simplifies to
	\begin{equation}\label{self 6h}
		\Delta^\perp A+\lambda\nabla^\perp_{\nabla s} A+(\vert A\vert ^2+\lambda)A=0
	\end{equation}
	from which we conclude
	\begin{equation}\label{self 6hb}
		\Delta\vert A\vert^2-2\vert\nabla^\perp A\vert^2 +\lambda\langle\nabla s,\nabla\vert A\vert^2\rangle+2(\vert A\vert ^2+\lambda)\vert A\vert^2=0.
	\end{equation}
	Finally, since
	\begin{eqnarray*}
		Q^2&:=&\vert \nabla^\perp A\otimes H-A\otimes\nabla^\perp H\vert^2\\
		&=&\vert\nabla^\perp A\vert^2|H|^2+|A|^2\vert\nabla^\perp H\vert^2-\frac{1}{2}\langle\nabla|A|^2,\nabla|H|^2\rangle
	\end{eqnarray*}
	we derive in a similar way to the computations in \cite{huisken} and \cite{mss}
	\begin{eqnarray}
		\Delta\frac{|A|^2}{|H|^2}+\left\langle\nabla\frac{|A|^2}{|H|^2},\lambda\nabla s+\frac{\nabla|H|^2}{|H|^2}\right\rangle-2\frac{Q^2}{|H|^4}=0.
	\end{eqnarray}
	If $\frac{|A|^2}{|H|^2}$ attains a local maximum, the strong elliptic maximum principle implies that $\frac{|A|^2}{|H|^2}$ is constant and that $Q^2$ vanishes. Hence
	\begin{equation}\label{q2}
		\nabla^\perp A\otimes H=A\otimes\nabla^\perp H. 
	\end{equation}
	Codazzi's equation then implies that the tensor $A\otimes\nabla^\perp H$ is fully symmetric, considered as a trilinear form an $TM$. We destinguish two cases.
	
	\textit{Case 1}. Suppose that $\nabla ^\perp H=0$ everywhere, i.e. that $|H|$ is constant. Since by assumption $M$ is not a linear subspace this constant cannot be zero by Remark \ref{rmk 1}(a). Equation \eqref{q2} then implies $\nabla^\perp A=0$. So all principle curvatures of $M$ are constant and due to a well known theorem of Lawson, it follows that $M$ is locally isometric to the product of a round sphere and an euclidean factor. Since those submanifolds are not self-expanding this is impossible and case 1 never occurs.
	
	\textit{Case 2}. Since case 1 is impossible we may therefore assume that there exists a simply connected domain $U\subset M$ where $\nabla^\perp H\neq 0$. On this set we choose an orthonormal frame 
	$$\{e_1:=\nabla|H|/|\nabla|H||,e_2,\dots,e_m\}.$$
	Then from the symmetry of $A\otimes\nabla^\perp H$ we obtain $A(e_j,e_k)=0$, for any $k\ge 1$ and $j\ge 2$. Therefore, $M$ has only one non-zero principal curvature on $U$ and $|A|^2=|H|^2$ on $U$. Let $\mathscr{D}:U\to TU$,
	$$\mathscr{D}(p):=\{v\in T_pU:A(v,\cdot)=0\},$$
	be the nullity distribution and $\mathscr{D}^\perp:=\operatorname{span}\{e_1\}$ its orthogonal complement. Exactly as in \cite{mss} we have $TU=\mathscr{D}\oplus \mathscr{D}^\perp$ and conclude that both distributions $\mathscr{D}$, $\mathscr{D}^\perp$ are parallel so that by the de Rham decomposition theorem, $U$ splits into the Riemannian product of a planar curve $\Gamma$ and an $(m-1)$-dimensional euclidean factor. Since $U$ is a self-expander, the curve $\Gamma$ must be part of a self-expanding curve. It is well known that self-expanding hypersurfaces are real analytic, therefore by completeness the local splitting implies the global splitting. This completes the proof.
\end{proof}
Recall that a hypersurface is called \textit{mean convex}, if $|H|>0$ everywhere. The next corollary shows that the pinching quantity $|A|^2/|H|^2$ on mean convex self-expanders is controlled by its asymptotic behavior at infinity.
\begin{corollary}\label{coro 0}
	Let $M\subset \real{m+1}$ be a properly immersed mean convex self-expanding hypersurface and suppose 
	$$\mu:=\lim_{r\to\infty}\sup_{M\setminus B(0,r)}\frac{|A|^2}{|H|^2}<\infty,$$
	where $B(0,r)$ denotes the closed euclidean ball of radius $r$ centered at the origin. Then one of the following cases holds:
	\begin{enumerate}[\rm (a)]
		\item $\frac{|A|^2}{|H|^2}<\mu$ on all of $M$.
		\item $M$ is a self-expanding hyperplane $\Gamma\times\real{m-1}$ and $\frac{|A|^2}{|H|^2}$ is constant to $\mu=1$.
	\end{enumerate}
	In particular, if $\mu\neq 1$, then {\rm(a)} holds.
\end{corollary}
Note, that $\mu=1$ does not exclude case (a).
\begin{proof}
	Since $M$ is mean convex, the function $f:=\frac{|A|^2}{|H|^2}$ is well defined on all of $M$.
	
	\textit{Step 1}. We will first prove that $f\le \mu$ on $M$. Suppose there exists a point $p\in M$ such that at $p$ we have $f(p)=\mu+\epsilon$ for some $\epsilon>0$. Choose $r>0$ such that $\sup_{M\setminus B(0,r)} f\le\mu+\epsilon/2$. Then $p\in M\cap B(0,r)$. Moreover, since $M$ by assumption is properly immersed, the set $K:=M\cap B(0,R)$ is compact. Hence the function $f$ attains a local maximum on $K$. From theorem \ref{main theo1} we conclude that $f$ is constant to $1$ and $M=\Gamma\times \real{m-1}$. Since $f(p)\neq \mu$ this gives a contradiction.
	
	\textit{Step 2}. From step 1 we know $\sup_Mf\le \mu$. Theorem \ref{main theo1} implies that either $f<\mu$ on all of $M$ or $f$ is constant to $1$ and $M$ is equal to a self-expanding hyperplane. This completes the proof.
\end{proof}
The last corollary can also be stated in the following form:

\begin{corollary}
	For any properly immersed mean convex self-expanding hypersurface $M\subset\real{m+1}$ we have 
	$$\frac{|A|^2}{|H|^2}\le\lim_{r\to\infty}\sup_{M\setminus B(0,r)}\frac{|A|^2}{|H|^2}$$
	and equality occurs at some point $p\in M$, if and only if $M=\Gamma\times\real{m-1}$ is a self-expanding hyperplane in which case $\frac{|A|^2}{|H|^2}$ is constant to $1$.
\end{corollary}
There exist some situations where the asymptotic behavior of $\frac{|A|^2}{|H|^2}$ can be easily controlled, e.g. if the self-expander is smoothly asymptotic to a cone.
\begin{corollary}\label{coro 1}
	Any properly immersed mean convex self-expanding surface $M^2\subset\real{3}$ that is smoothly asymptotic to a cone must be strictly convex.
\end{corollary}
\begin{proof}
	On any $2$-dimensional cone the function $\frac{|A|^2}{|H|^2}$ is constant to $1$, therefore we conclude 
	$$\mu=\lim_{r\to\infty}\sup_{M\setminus B(0,r)}\frac{|A|^2}{|H|^2}=1.$$
	The statement now follows from Gau\ss' equation for the scalar curvature $S=|H|^2-|A|^2$ and from Corollary \ref{coro 0} since the self-expanding hyperplane is not smoothly asymptotic to a cone.
\end{proof}
There exists an extension of Corollary \ref{coro 1} to any dimension in the following sense.
\begin{corollary}\label{coro 2}
	Any properly immersed mean convex self-expanding hypersurface $M\subset\real{m+1}$ that is smoothly asymptotic to a cone with nonnegative scalar curvature must attain strictly positive scalar curvature.
\end{corollary}
\begin{proof}
	Since the cone $C$ has nonnegative scalar curvature $S$, we must have $S=|H|^2-|A|^2\ge 0$ on $C$. On the other hand $M$ is mean convex and asymptotic to $C$ so that the cone must be mean convex as well. Therefore $\frac{|A|^2}{|H|^2}\le 1$ on $C$ and
	$$\mu=\lim_{r\to\infty}\sup_{M\setminus B(0,r)}\frac{|A|^2}{|H|^2}\le 1.$$
	Again from Corollary \ref{coro 0} and since the self-expanding hyperplane is not smoothly asymptotic to a cone we conclude $\frac{|A|^2}{|H|^2}< 1$ on $M$ which by Gau\ss' equation is equivalent to the statement that $M$ has strictly positive scalar curvature.
\end{proof}
The second fundamental form of a hypersurface that is mean convex and of positive scalar curvature satisfies some nice properties.
\begin{lemma}\label{lemm alg}
	Let $M\subset\real{m+1}$ be mean convex with positive scalar curvature $S$. Then the principal curvatures $\lambda_1,\dots,\lambda_m$ of $M$ satisfy $\lambda_i<|H|$ for $i=1,\dots,m$. 
\end{lemma}
\begin{proof}
	The proof is purely algebraic. The mean and scalar curvatures of $M$ are related to the principal curvatures $\lambda_1,\dots,\lambda_m$ by
	$$|H|=\sum_k\lambda_k$$
	and
	$$S=|H|^2-|A|^2=2\sum_{k<l}\lambda_k\lambda_l.$$
	Suppose there exists $i\in\{1,\dots,m\}$ with $|H|-\lambda_i\le0$. Then we can estimate
	\begin{eqnarray*}
		\frac{S}{2}&=&\lambda_i(|H|-\lambda_i)+\sum_{\substack{k<l\\k,l\neq i}}\lambda_k\lambda_l\le\sum_{\substack{k<l\\k,l\neq i}}\lambda_k\lambda_l.
	\end{eqnarray*}
	On the other hand
	\begin{eqnarray*}
		\frac{S}{2}&=&\lambda_i(|H|-\lambda_i)+\sum_{\substack{k<l\\k,l\neq i}}\lambda_k\lambda_l\\
		&=&-(|H|-\lambda_i)^2+|H|(|H|-\lambda_i)+\sum_{\substack{k<l\\k,l\neq i}}\lambda_k\lambda_l\\
		&&\le-(|H|-\lambda_i)^2+\sum_{\substack{k<l\\k,l\neq i}}\lambda_k\lambda_l\\
		&=&-\sum_{k\neq i}\lambda_k^2-2\sum_{\substack{k<l\\k,l\neq i}}\lambda_k\lambda_l+\sum_{\substack{k<l\\k,l\neq i}}\lambda_k\lambda_l\le -\sum_{\substack{k<l\\k,l\neq i}}\lambda_k\lambda_l.
	\end{eqnarray*}
	Together this implies
	$$\frac{S}{2}\le-\left|\sum_{\substack{k<l\\k,l\neq i}}\lambda_k\lambda_l\right|\le0$$
	and this is a contradiction.
\end{proof}
A hypersurface $M\subset\real{m+1}$ is called $k$-convex, if at each point $p\in M$ the sum of any $k$ of the $m$ principal curvatures $\lambda_1,\dots,\lambda_m$ is positive. Obviously, $m$-convexity is the same as mean convexity and a strictly convex hypersurface is $1$-convex.
Therefore Corollary \ref{coro 2} and Lemma \ref{lemm alg} imply
\begin{corollary}\label{coro 3}
	Any properly immersed mean convex self-expanding hypersurface $M\subset\real{m+1}$ that is smoothly asymptotic to a cone with nonnegative scalar curvature must be $(m-1)$-convex.
\end{corollary}
Since a $3$-dimensional cone in $\real{4}$ has nonnegative scalar curvature, if it is convex, we conclude in particular
\begin{corollary}\label{coro 4}
	Any properly immersed mean convex self-expanding hypersurface $M\subset\real{4}$ that is smoothly asymptotic to a convex cone is $2$-convex.
\end{corollary}
There exist more results that can be obtained from Theorem \ref{main theo1} and its corollaries, for example
\begin{corollary}\label{coro 5}
	Any properly immersed mean convex self-expanding surface $M\subset\real{3}$ that is smoothly asymptotic to a self-expanding hyperplane $\Gamma\times\real{}$ is a self-expanding hyperplane.
\end{corollary}
\begin{proof}
	Since $M$ is smoothly asymptotic to $\Gamma\times\real{}$ we have 
	$$\mu=\lim_{r\to\infty}\sup_{M\setminus B(0,r)}\frac{|A|^2}{|H|^2}=1.$$
	Corollary \ref{coro 0} implies that $M$ is either equal to a self-expanding hyperplane or strictly convex. Since there do not exist strictly convex surfaces smoothly asymptotc to the flat product $\Gamma\times\real{}$, only the first case will be possible.
\end{proof}
However, one should note that a simple scaling argument shows that self-expanding hyperplanes are not necessarily equal, if they are asymptotic to each other. 
\section{Self-expanders in higher codimension}\label{section 6}
Now we will extend Theorem \ref{main theo1} to the case where $M^m\subset\real{n}$ is a self-expander in higher codimension. The idea is to study the same quantity as in \cite{smoczyk} for self-shrinkers. Let $A^H=\langle A,H\rangle$ be the second fundamental form with respect to the mean curvature vector $H$. Instead of considering the quotient $|A|^2/|H|^2$ as in the last chapter, we treat the scaling invariant quotient $|A^H|^2/|H|^4$ which for hypersurfaces coincides with $|A|^2/|H|^2$. As in \cite{smoczyk} we will see that his quantity has a much better behavior. In addition, in this section we will always assume that $|H|>0$ and that the
\textit{principal normal vector field}
$$\xi:=\frac{H}{|H|}$$
is parallel in the normal bundle, i.e.
\begin{equation}\label{eqcod0}
	\nabla^\perp\xi=0.
\end{equation}
This condition is redundant for hypersurfaces but turns out to be crucial in the forthcoming computations. Consequently we have
\begin{equation}\label{eqcod1}
	\nabla^\perp H=\nabla|H|\otimes \xi,\quad \Delta^\perp H=\Delta|H|\cdot\xi.
\end{equation}
The computations in \cite{smoczyk} for self-shrinkers carry over almost unchanged, in particular Lemma 3.3 in \cite{smoczyk} now becomes
\begin{lemma}
	Let $M^m\subset\real{n}$ be a self-expander with $|H| > 0$ and parallel principal normal $\xi$. Then the following equation holds.
	\begin{eqnarray}
		\Delta\frac{|A^H|^2}{|H|^4}&=&\frac{2}{|H|^4}\left|\nabla|H|\otimes\frac{A^H}{|H|}-|H|\nabla\frac{A^H}{|H|}\right|^2\nonumber\\
		&&-\lambda\left\langle\nabla s,\nabla\frac{|A^H|^2}{|H|^4}\right\rangle-\frac{2}{|H|}\left\langle\nabla |H|,\nabla\frac{|A^H|^2}{|H|^4}\right\rangle.\label{eq ah}
	\end{eqnarray}
\end{lemma}
In the sequel we will need the following operator. Let $E,F$ be two vector bundles over $M$ and suppose $C\in\Gamma(E\otimes T^*M\otimes T^*M)$ and $D\in\Gamma(F\otimes T^*M\otimes T^*M)$ are two bilinear forms with values in the vector bundles $E$ respectively $F$. For example $C$ could be the bilinear form $A^H$ (in which case $E$ is the trivial bundle) or $D$ could be the second fundamental tensor $A\in\Gamma(T^\perp M\otimes T^*M\otimes T^*M)$. Then $C\circledast D\in\Gamma(E\otimes F\otimes T^*M\otimes T^*M)$ is by definition the bilinear form given by the trace
$$(C\circledast D)(v,w):=\sum_{k=1}^mC(v,e_k)\otimes D(e_k,w),$$
where $e_1,\dots, e_m$ is an arbitrary orthonormal frame in $TM$.
\begin{theorem}\label{main theo 2}
	Let $M^m\subset\real{n}$ be a complete and connected self-expander with $H\neq 0$, bounded second fundamental form $A$ and parallel principal normal $\xi=H/|H|$. Then the following statements are equivalent:
	\begin{enumerate}[\rm (a)]
		\item $M$ is a self-expanding hyperplane $\Gamma\times\real{m-1}$.
		\item The function $\frac{\vert A^H\vert^2}{\vert H\vert ^4}$ attains a local maximum.
	\end{enumerate}
	If one of these equivalent conditions is satisfied, then  $\frac{|A^H|^2}{|H|^4}$ is constant to $1$.
\end{theorem}
\begin{proof}
	The proof will be separated into several steps.
	\begin{enumerate}[\rm(i)]
		\item First note that 
		\begin{equation}\label{eqt 1}
			A^H\circledast A=A\circledast A^H.
		\end{equation}
		This is a consequence of Ricci's equation in Proposition \ref{prop 1}(b) and of $\nabla^\perp\xi=0$, because
		\begin{eqnarray*}
			0&=&|H|R^\perp(v,w)\xi=R^\perp(v,w)H\\
			&=&(A\circledast A^H-A^H\circledast A)(v,w).
		\end{eqnarray*}
		\item The strong elliptic maximum principle and equation \eqref{eq ah} imply that
		\begin{equation}\label{eqt 2}
			\frac{|A^H|^2}{|H|^4}=c
		\end{equation}
		for some constant $c>0$ and
		\begin{equation}\label{eqt 3}
			\nabla|H|\otimes\frac{A^H}{|H|}-|H|\nabla\frac{A^H}{|H|}=0
		\end{equation}
		From Codazzi's equation and since $\xi$ is parallel we obtain that
		$\nabla^\perp\frac{A^H}{|H|}=\nabla^\perp A^\xi$ is fully symmetric. Then as in \cite{smoczyk} we can decompose the quantity on the LHS in \eqref{eqt 3} into its symmetric and anti-symmetric parts to derive that $\nabla|H|\otimes A^H$ is fully symmetric and therefore
		\begin{equation}\label{eqt 4}
			|A^H|^2|\nabla|H||^2-(A^H\circledast A^H)(\nabla|H|,\nabla|H|)=0.
		\end{equation}
		
		\item We will destinguish two cases.\\
		
		\noindent\textit{Case 1.} Suppose that $\nabla|H|=0$ on $M$ which in view of $\nabla^\perp\xi=0$ is equivalent to $\nabla^\perp H=0$. Then $\lambda>0$ and equation \eqref{self 4b} show that $|H|=0$  which is a contradiction to our assumption  (in fact, the same equation shows that on any self-expander the function $|H|$ cannot attain local positive minima). So this case cannot occur.\\
		
		\noindent\textit{Case 2.} From the full symmetry of the tensor $\nabla|H|\otimes A^H$ that we obtained in step (ii) one derives that at a point $p\in M$ where $\nabla|H|(p)\neq 0$ any tangent vector $v\in T_pM$ orthogonal to $\nabla|H|(p)$ is a zero eigenvector of $A^H$ at $p$ and  that $\nabla|H|$ is an eigenvector of $A^H$ to the eigenvalue $|H|^2$ (since $\operatorname{trace}(A^H)=|H|^2$).  In particular, the tensor $A^H$ has only one non-zero eigenvalue and $|A^H|^2=|H|^4$ on all of $M$. Thus as in \cite{smoczyk} on the open set
		$$M^o:=\{p\in M:\nabla |H|(p)\neq 0\}$$
		we define the two distributions
		$$\mathscr{E}_pM^o:=\{v\in T_pM^o:A^H(v,\cdot)=|H|^2\langle v,\cdot\rangle\},$$
		$$\mathscr{F}_pM^o:=\{v\in T_pM^o:A^H(v,\cdot)=0\}.$$
		Taking into account Theorem 3.20 in \cite{gssz} or Lemma 6.4 in \cite{halldorsson}, we may then proceed exactly as in \cite{smoczyk} to prove that $\mathscr{E}$, $\mathscr{F}$ can be smoothly extended to parallel distributions on all of $M$ and that $M$ splits into the Riemannian product $M=\Gamma\times\real{m-1}$, where $\Gamma$ is a self-expanding curve, and that the distributions $\mathscr{E}, \mathscr{F}$ form the tangent bundles of $\Gamma$ respectively $\real{m-1}$. 
	\end{enumerate}
	This completes the proof of Theorem \ref{main theo 2}.
\end{proof}


%
%

\bibliographystyle{spbasic}

\end{document}